\documentclass[10pt]{amsart}
\usepackage{amsmath,amssymb,amsthm,url}
\usepackage{graphicx}
\usepackage{color}
\newtheorem{theorem}{Theorem}    
\newtheorem{lemma}{Lemma} 
 
\newtheorem{proposition}{Proposition} 
\newtheorem{conjecture}{Conjecture} 
\newtheorem{question}{Question} 
\newtheorem{corollary}{Corollary}
\theoremstyle{definition}
\newtheorem{definition}[theorem]{Definition}

\newtheorem{remark}[theorem]{Remark}
\newtheorem*{remark*}{Remark}
\newcommand{\Z}{\mathbb{Z}}

\title{Genus non-increasing totally positive unknotting number}
\author{Tetsuya Ito}
\address{Department of Mathematics, Kyoto University, Kyoto 606-8502, JAPAN}
\email{tetitoh@math.kyoto-u.ac.jp}
\begin{document}

\begin{abstract}
The genus non-increasing totally positive unknotting number is the minimum number of crossing changes that transform a knot into the unknot, such that all the crossing changes are positive-to-negative crossing changes that do not increase the genus.
We show that the genus non-increasing totally positive unknotting number can be arbitrary large for genus one knots.
\end{abstract} 

\maketitle

\section{Introduction}

The \emph{unknotting number} $u(K)$ of a knot $K$ is the minimum number of crossing changes needed to convert $K$ to the unknot. Despite its simple definition, the unknotting number is one of the most difficult invariants in knot theory.

Typically to determine the unknotting number, we first find a unknotting sequence, a sequence of crossing change converting $K$ to the unknot. Then we confirm that the length of the sequence is the minimum by lower bounds or obstructions of the unknotting number.
 
However, in almost all cases, known lower bounds or criterion can show $2g(K)\leq u(K)$ at best.
For example, thanks to the the inequality $g_4(K) \leq u(K)$ of the slice genus $g_4(K)$, lower bounds of slice genus $g_4(K)$ from concordance invariants give a lower bound of unknotting number. However, such lower bounds are not effective when $u(K)>g(K)$ because $g_4(K)\leq g(K)$.

Similarly, for classical unknotting obstructions coming from the $S$-equivalence class of Seifert matrix (we refer to \cite{bf2} for concise summaries of classical and some other lower bounds of unknotting numbers), they actually give a lower bound of \emph{algebraic unknotting number} $u_{a}(K)$, the minimum number of crossing change to get a knot with trivial Alexander polynomial \cite{bf,bf2}. The algebraic unknotting number is bounded above by $2g(K)$ \cite[Corollary 4]{fl}, so classical unknotting obstructions are not effective when $u(K)>2g(K)$.

Knot homology theories provide more obstructions of unknotting numbers \cite{al,ae,ow,os} that do not belong to the above categories. Among them, Owen's criterion shows that the knot $K=9_{35}$, the pretzel knot $P(3,3,3)$ satisfies $u(K)=3>2=2g(K)$ \cite{ow}. 

Nevertheless, it is still fair to say that it is extremely difficult to determine the unknotting number when $u(K)>2g(K)$. According to Knotinfo \cite{lm}, among prime knots with crossing number $\leq 13$, the aforementioned knot $9_{35}$ is the unique known example of knots with $u(K)>2g(K)$.

\begin{question}
\label{question:main}
For a given $g>0$, is the set $\{ u(K) \: | \: g(K)\leq g\}$ unbounded ?
In particular, is the set $\{u(K) \: | \: g(K)=1 \}$ unbounded ?
\end{question}

We expect the answer is affirmative. For example, the 3-strand pretzel knots $P(2p+1,2q+1,2r+1)$ has genus one and it is expected that $u(P(2p+1,2q+1,2r+1)) \to \infty$ as $p,q,r \to \infty$. However, it is not even known whether $u(P(2p+1,2q+1,2r+1)) \geq 4$ holds for sufficiently large $p,q,r$.

The aim of this note is to give a positive answer to Question \ref{question:main} for a certain variant of the unknotting number.

A \emph{positive unknotting sequence} is a sequence of positive-to-negative crossing changes that makes $K$ the unknot. The \emph{totally positive unknotting number} $u_{++}(K)$ is the minimum length of positive unknotting sequence. Here by convention if $K$ has no positive unknotting sequence we define $u_{++}(K)=+\infty$. 
We study the following variant of totally positive unknotting number.

\begin{definition}
The \emph{genus non-increasing totally positive unknotting number} $u^{g}_{++}(K)$ of a knot $K$ is the minimum length of positive unknotting sequence of $K$ such that the genus is non-increasing.
\end{definition}

Although at first glance, the restriction looks too strong, but many knots have (conjecurally) minimum unknotting sequence that is both positive and genus non-increasing. For example, when $K$ is a positive braid knot, or more generally, a knot which is a plumbing of trefoils, $u^{g}_{++}(K)=u_{++}(K)=u(K)=g_4(K)=g(K)$ holds \cite{klmmms}.
The aforementioned pretzel knot $P(2p+1,2q+1,2r+1)$ has a positive and genus non-increasing unknotting sequence of length $\min\{p+q,q+r,r+p\}+1$ when $p,q,r>0$, and conjecturally this attains the unknotting number.

Our main theorem gives a lower bound of $u^{g}_{++}(K)$ for a genus one knot $K$ in terms of the HOMFLY polynomial.

\begin{theorem}
\label{theorem:main}
Let $K$ be a genus one knot and $P_K(v,z)=\sum_{i,j} h_{i,j}v^{2i}z^{2j} \in \Z[v,z]$ be the HOMFLY polynomial of $K$.
\begin{itemize}
\item[(i)] Let $m = \min \{i \: | \: h_{i,0} \neq 0\}$.  
\begin{itemize}
\item[(i-a)] $u^{g}_{++}(K)\geq m$.
\item[(i-b)] If $h_{m,0}<0$ then $u^{g}_{++}(K)=\infty$. 
\item[(i-c)] If $h_{m,0}=1$ and $h_{m',0}<0$, then $u^{g}_{++}(K)\geq m+1$. Here $m'= \min \{i>m \: | \: h_{i,0} \neq 0\}$.
\end{itemize}
\item[(ii)] Let $M = \max \{i \: | \: h_{i,0} \neq 0 \}$. 
\begin{itemize}
\item[(ii-a)] If $h_{M,0}>0$, then $u^{g}_{++}(K) \geq M$. 
\item[(ii-b)] If $h_{M,0}>1$, then $u^{g}_{++}(K) \geq M+1$.
\end{itemize}
\end{itemize}
\end{theorem}

Our proof of Theorem \ref{theorem:main} is based on our previous observation that the HOMFLY polynomial gives an obstruction for two genus one knots has Gordian distance one, namely, whether they are related by a single crossing change or not \cite{it}. 

In general, a repeated use of a Gordian distance one obstruction brings us to a lower bound of unknotting numbers, but it becomes much weaker and complicated when we try to treat large unknotting numbers. As we will see, it is remarkable that our Gordian distance one obstruction provides a simple and effective criterion for large unknotting numbers. 

For the genus one pretzel knot $K=P(2p+1,2q+1,2r+1)$, its HOMFLY polynomial is given by
\[ P_K(v,z)= v^{2(p+q+1)}+v^{2(q+r+1)}+v^{2(r+p+1)}-v^{2(p+q+r+1)}-v^{2(p+q+r+2)} + z^2 p^1_{K}(v)\]
where $p^1_K(v) \in \Z[v,v^{-1}]$. Hence by Theorem \ref{theorem:main} (i-a) we get the following.

\begin{corollary}
Let $p,q,r>0$. For the genus one pretzel knot $P(2p+1,2q+1,2r+1)$, $u^{g}_{++}(P(2p+1,2q+1,2r+1)) = \min\{p+q,q+r,r+p\}+1$.
\end{corollary}
Consequently, we get an affirmative answer to Question \ref{question:main} for  $u^{g}_{++}$.

\begin{corollary}
For genus one knots, $u^{g}_{++}$ is unbounded.
\end{corollary}

Our result gives a supporting evidence for the conjecture $u(P(2p+1,2q+1,2r+1))= \min\{p+q,q+r,r+p\}+1$ ($p,q,r>0$). In particular, it says that if $u(P(2p+1,2q+1,2r+1))$ is smaller than the conjectural value, then a minimum unknotting sequence must contain a negative-to-positive crossing change, or, go thorough higher genus knots.

For the positive genus one pretzel knot case, we expect $u^g_{++}(K)=u(K)$ but in general these values can are different even if $u^g_{++}(K) < \infty$. Let $T_{2m}$ be the $2m$-twist knot. It is easy to see that $T_{2m}$ has unknotting number one. $T_{2m}$ can be unknotted by a single negative-to-positive crossing change. On the other hand, it also has a genus non-increasing positive unknotting sequence of length $m$. 
By direct computation, we see that the $2m$-twist knot $T_{2m}$ has HOMFLY polynomial
\[ P_{T_{2m}} (v,z)=v^{-2}-v^{2m-2}+v^{2m} + z^2 p^1_{T_{2m}}(v)\]
where $p^1_{T_{2m}}(v) \in \Z[v,v^{-1}]$. Thus by Theorem \ref{theorem:main} (ii-a) we get the following.

\begin{corollary}
For the $2m$-twist knot $K=T_{2m}$ $(m>0)$, $u(K)=1$ but $u^{g}_{++}(K)=m$. Thus the difference of $u(K)$ and $u^{g}_{++}(K)$ can be arbitrary large (even for genus one knots).
\end{corollary}

%
%

\section*{Acknowledgements}
The author is partially supported by JSPS KAKENHI Grant Numbers 19K03490, 21H04428, 	23K03110. The author thanks M. Borodzik for helpful correspondence.

\section{Unknotting sequence with restrictions}

In this section we give a brief discussion for unknotting sequence with prescribed restrictions. 
 
In the following, for knots $K$ and $K'$ we denote by $K \stackrel{\varepsilon}{\longrightarrow} K'$ $(\varepsilon \in \{\pm 1\})$ if $K'$ is obtained from $K$ by a crossing change of the sign $\varepsilon$. 

\emph{An unknotting sequence} of a knot $K$ is a sequence $\mathcal{K}$ consisting of crossing changes that makes $K$ into the unknot $U$
\[ K=K_n \stackrel{\varepsilon_n}{\longrightarrow} K_{n-1} \stackrel{\varepsilon_{n-1}}{\longrightarrow} \cdots  \stackrel{\varepsilon_2}{\longrightarrow} K_1 \stackrel{\varepsilon_1}{\longrightarrow} U \]
The \emph{length} $\ell(\mathcal{K})$ of the unknotting sequence $\mathcal{K}$ defined by $\ell(\mathcal{K})=n$. 

Let $v$ be an invariant of knot that takes value in an ordered set such that the value of the unknot $v(U)$ is minimum. Typically, like the knot genus $g(K)$, we consider an invariant that takes value in non-negative integers or natural numbers.

We say that an unknotting sequence $\mathcal{K}$ is
\begin{itemize}
\item[--] \emph{positive} if $\varepsilon_i = + $ for all $i$, i.e., it consists of positive-to-negative crossing change.
\item[--] \emph{$v$ non-increasing} if $v(K_{i+1})\geq v(K_{i})$ for all $i$.
\end{itemize}

By considering unknotting sequences with particular properties, we have several variants of unknotting numbers.
\begin{definition}
For a knot $K$, we define 
\begin{itemize}
\item the \emph{totally positive unknotting number} 
\[ u_{++}(K) = \min \{\ell(\mathcal{K}) \: | \: \mathcal{K} \mbox{ is a positive unknotting sequence of }K\}.\]
\item the \emph{$v$ non-increasing unknotting number}
\[ u^v(K) =\min \{\ell(\mathcal{K}) \: | \: \mathcal{K} \mbox{ is a } v \mbox{ non-increasing  unknotting sequence of }K\}.\]
\item the \emph{$v$ non-increasing unknotting number}
\[ u^v_{++}(K) =\min \{\ell(\mathcal{K}) \: | \: \mathcal{K} \mbox{ is a } v \mbox{ non-increasing positive unknotting sequence of }K\}.\]
\end{itemize}
\end{definition}
Here we use the convention that $\min \emptyset = +\infty$.

\begin{remark}
There are many variants of unknotting numbers that take into accounts of signs. The most standard one is the \emph{positive unknotting number} $u_{+}(K)$, the minimum number of positive-to-negative crossing change among the minimum unknotting sequence of $K$ (see, for example, \cite{cl}). We refer to \cite{tr} for other variants of signed unknotting numbers. 
\end{remark}

We observe that for a natural knot invariants $v$ coming from a complexity of diagrams, every knot $K$ admits a $v$ non-increasing unknotting sequence.
\begin{theorem}
Let $v$ be a knot invariant that satisfies the following properties.
\begin{itemize}
\item[(a)] $v$ can be defined as the minimum values of invariant of knot diagrams $v_D$. There exists an invariant $v_d$ of knot diagrams such that 
\[ v(K)=\min\{v_d(D) \: | \: D \mbox{ is a diagram of }K\}\]
\item[(b)] The diagram invariant $v_d$ does not change by the crossing change. If $D'$ is obtained from $D$ by a crossing change at a crossing $c$ of $D$, then $v_d(D)=v_d(D')$. In other words, $v_d(D)$ can be regarded as an invariant of shadows, a knot diagram ignoring its over-under information.
\item[(c)] $v(K)\geq v(U)$ for all knots $K$ and $v(K)=v(U)$ holds if and only if $K=U$. 
\end{itemize}
Then every knot $K$ admits a $v$ non-increasing unknotting sequence, i.e., $u^{v}(K)$ is finite.
\end{theorem}
\begin{proof}
We prove the assertion by induction on $v(K)$. Thanks to the property (c), $v(K)=v(U)$ case is obvious.
 
Let $D$ be a diagram of $K$ that attains the minimum so $v_d(D)=v(K)$. Since every knot diagram $D$ can be converted to a diagram representing the unknot $U$ by applying crossing change, there is a sequence of diagrams
\[ D =D_n \stackrel{\varepsilon_n}{\longrightarrow} D_{n-1} \stackrel{\varepsilon_{n-1}}{\longrightarrow} \cdots \stackrel{\varepsilon_2}{\longrightarrow} D_1  \stackrel{\varepsilon_1}{\longrightarrow} D_0 \]
such that $D_0$ represents the unknot $U$.
Here $D_i \stackrel{\varepsilon_i}{\longrightarrow} D_{i-1}$ means that the diagram $D_{i-1}$ is obtained from from the diagram $D_{i}$ by changing a crossing of sign $\varepsilon_{i}$.
 
Let $K_i$ be the knot represented by $D_i$.
By the property (b) $v_d(D)=v_d(D_n)=\cdots=v_d(D_1)=v_d(D_0)$ so by the property (a)
 \[ v(D_{n}) =v(K) \geq v(K_1),\ldots,v(K_{n})\geq v(K_0)=v(U)\]
Thus there is $i\geq 0$ such that $v(K)=v(K_n)=\cdots = v(K_{i+1})>v(K_i)$ holds.
By induction, there is a $v$ non-increasing unknotting sequence of $K_i$ hence $K$ has a $v$ non-increasing unknotting sequence.
\end{proof}

\begin{corollary}
For the crossing number $c$, the braid index $\mathsf{braid}$, the bridge index $\mathsf{bridge}$ and the canonical genus $g_c$, \footnote{The canonical genus $g^{c}$ is the minimum genus of a Seifert surface obtained by Seifert's algorithm.}, $u^{c},u^{\sf braid}, u^{\sf bridge},u^{g_c}$ are always finite.
\end{corollary}

We expect crossing number non-increasing unknotting number is always equal to $K$.

\begin{conjecture}
\label{conj:ug=u}
$u^{c}(K)=u(K)$.
\end{conjecture}

The Bernhard-Jablan unknotting conjecture states that every knot $K$ has a minimum crossing diagram $D$ with a crossing change that yields a knot with $u(K')=u(K)-1$. Although this obviously implies $u^{c}(K)=u(K)$, unfortunately, Bernhard-Jablan unknotting conjecture is known to be false in general \cite{bh}. Conjecture  \ref{conj:ug=u} can be seen as a modification of Bernhard-Jablan unknotting conjecture.

On the other hand, we do not know whether the \emph{genus non-increasing unknotting number} $u^{g}(K)$ is finite in general.
\begin{question}
Does every knot $K$ admits genus non-increasing unknotting sequence?
\end{question}

We just observe that we can unknot a knot $K$ without increasing the genus so much.

\begin{proposition}
For every knot $K$, there is an unknotting sequence $\mathcal{K}$
\[ K=K_n \stackrel{\varepsilon_n}{\longrightarrow} K_{n-1} \stackrel{\varepsilon_{n-1}}{\longrightarrow} \cdots  \stackrel{\varepsilon_2}{\longrightarrow} K_1 \stackrel{\varepsilon_1}{\longrightarrow} U \]
such that $g(K_i)\leq g(K)+2$ for all $i$.
\end{proposition}
\begin{proof}

Let $S$ be a minimum genus Seifert surface of $K$ which we view as as a 2-disk attached $2g(K)$ 1-handles. We put the surface $S$ so that the 1-handles are projected to a parallel of tangles as in Figure \ref{fig:surface}.

\begin{figure}[htbp]
\begin{center}
\includegraphics*[width=70mm]{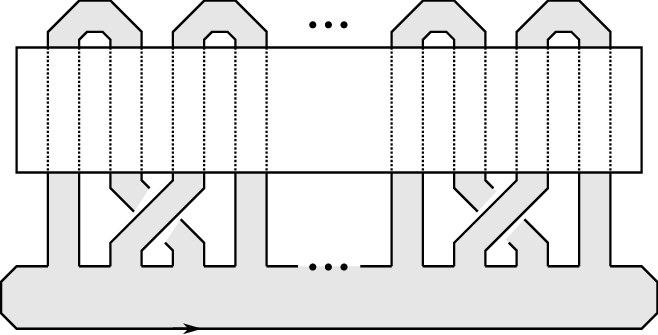}
\begin{picture}(0,0)
\end{picture}
\caption{A knot $K$ and minimum genus Seifert surface. A box represents a parallel of tangle.} 
\label{fig:surface}
\end{center}
\end{figure} 

From this diagram, there is a sequence of knots $K=K_n \rightsquigarrow K_{n-1} \rightsquigarrow \cdots  \rightsquigarrow K_1 \rightsquigarrow K_{0}= U $
where $K_{i+1} \rightsquigarrow K_i$ means that $K_{i}$ is obtained from $K_{i+1}$ by a crossing change that changes the framing of 1-handles (Figure \ref{fig:band-pass} (a)), or, a band-pass move that changes the over-under information of $1$-handles (Figure \ref{fig:band-pass} (b)).
Since these moves can be seen as a move of surfaces, $g(K_{i+1}) \geq g(K_{i})$ holds for all $i$.

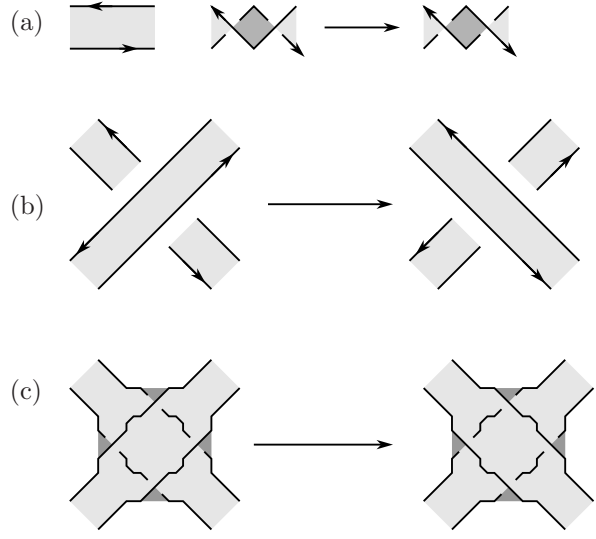
\begin{figure}[htbp]
\begin{center}
\includegraphics*[width=70mm]{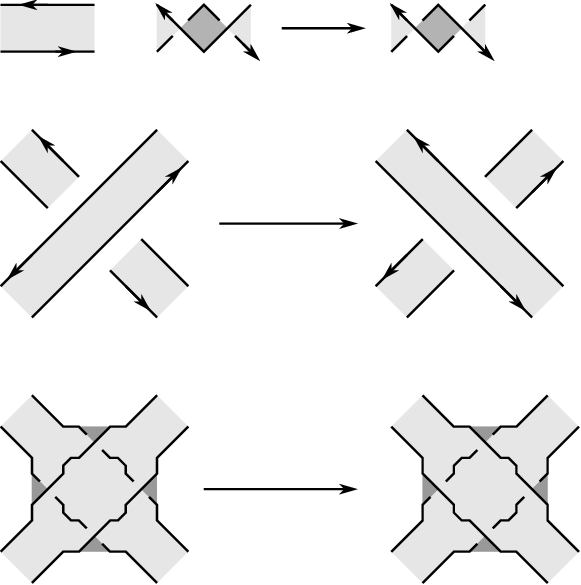}
\begin{picture}(0,0)
\put(-225,190) {(a)}
\put(-225,120) {(b)}
\put(-225,50) {(c)}
\end{picture}
\caption{(a) crossing change (b) band-pass move (c) band-pass move realized by four crossing changes that increases the genus by at most two} 
\label{fig:band-pass}
\end{center}
\end{figure} 

For each band-pass move, by slightly changing the surface $S$ near the crossing of bands, we see that the band-pass move $K_{i+1} \rightsquigarrow K_i$ is realized by four crossing changes
\[ K_{i+1}  \stackrel{\varepsilon_1}{\longrightarrow} K'_i  \stackrel{\varepsilon_1}{\longrightarrow} K''_i \stackrel{\varepsilon_1}{\longrightarrow} K'''_i  \stackrel{\varepsilon_1}{\longrightarrow} K_{i+1}\]
so that $g(K'_i),g(K''_i),g(K'''_i)\leq g(K_i)+2$ (Figure \ref{fig:band-pass} (c)). This shows that existence of unknotting sequence that increases the genus by at most two.    
\end{proof}

Besides the signs and invariant non-increasing conditions, it is also interesting to discuss unknotting sequence within a particular class of knots. For example, in \cite{klmmms} unknotting sequence within positive fibered knot is studied. They showed that, for a positive fibered knot $K$, the minimum length of unknotting sequence that consists of positive fibered knots is not equal to the unknotting number.

\section{HOMFLY polynomial obstruction of genus non-increasing unknotting number}
\label{section:homfly}

Let $P_K(v,z)$ be the HOMFLY polynomial of a knot or link $K$, defined by the skein relation
\[ v^{-1} P_{K_+}(v,z) -  vP_{K_-}(v,z)= z P_{K_0}(v,z), \quad P_{\sf Unknot}(v,z) =1. \]
From the skein relation one can check that $(v^{-1}z)^{\#K-1}P_K(v,z) \in \Z[v^{-2},v^2][z^{2}]$ so we may write
\[ P_K(v,z)=(v^{-1}z)^{-\#K + 1} \sum_{i=0} p^{i}_K(v)z^{2i}, \qquad p^{i}_K(v) \in \Z[v^2,v^{-2}]. \]
Here we denote by $\#K$ the number of components of $K$. 
We call the polynomial $p^{i}_K(v)$ the \emph{$i$-th coefficient (HOMFLY) polynomial} of $K$.

It is easy to check that for a knot $K$, $p^{0}_K(1)=1$ and $\frac{dp^0_K}{dv}(1)=0$ hold. Conversely, it is known that every polynomial $f(v) \in \Z[v^2,v^{-2}]$ with $f(1)=1$ and $f'(1)=0$ can be realized as the zeroth coefficient polynomial of some knots \cite{ka}.

The skein relation tells us that the zeroth coefficient polynomial $p^{0}_K(v)$ satisfies the following remarkable and simpler skein relations
\begin{equation}
\label{eqn:skein}
v^{-2} p^0_{K_+}(v) - p^{0}_{K_-}(v)= 
\begin{cases} p^0_{K_0}(v) & \mbox{if } \delta = 0,\\
0 & \mbox{if } \delta=1. 
\end{cases} 
\end{equation} 
where we define $\delta=0$ if two strands in the skein crossing belong to the same component, and $\delta=1$ otherwise. 
This says that for an $n$-component link $L=K_1 \cup \cdots \cup K_n$, its $0$-th coefficient polynomial is given by
\begin{equation}
\label{eqn:link}
p^{0}_{L}(v)= (v^{-2}-1)^{n-1}v^{2\sum_{i<j}lk(K_i,K_j)} p^{0}_{K_1}(v)p^{0}_{K_2}(v)\cdots p^{0}_{K_n}(v)
\end{equation}

Let $a_2(K)$ be the coefficient of $z^{2}$ in the Conway polynomial $\nabla_K(z)$ of a knot $K$. In \cite{it} we gave the following obstruction for genus one knots to related by a single crossing change.

\begin{theorem}\cite[Theorem 1.1]{it}
\label{theorem:Gordian}
Let $K$ be a genus one knot. Assume that a knot $K'$ is obtained from $K$ by a crossing change at a non-nugatory crossing $c$ of sign $\varepsilon$. If $g(K')\leq 1$, then there exists $f(v) \in \Z[v^{2},v^{-2}]$ such that
\[v^{-\varepsilon}p^{0}_K(v) - v^{\varepsilon}p^{0}_{K'}(v) = \varepsilon (v^{-1}-v)v^{2\varepsilon(a_2(K)-a_2(K'))}f(v)^2 \]
and that $f(1)=1$ and $f'(1)=0$.
\end{theorem}

Although in \cite{it} we stated the theorem for the assumption that $g(K)=g(K')=1$, the same argument works for the weaker assumption that $g(K') \leq g(K)=1$. For convenience of the reader, we repeat the proof.
\begin{proof}
We prove the theorem for the case $\varepsilon=+$ since $\varepsilon=-$ case is similar.

Let $K_0$ be the link obtained by smoothing the crossing $c$ and let $D$ be its crossing disk, a disk whose interior intersects with $K$ at two points that is the preimage of the crossing $c$, such that $lk(K,\partial D)=0$ (Figure \ref{fig:crossing} (a)). Let $S$ be a minimum genus oriented surface in $S^{3} \setminus \partial D$ that bounds $K$. By \cite[Theorem 2.1]{kl}. $g(S)=\max\{g(K),g(K')\}$ holds. Since we are assuming that $g(K') \leq g(K)=1$, it follows that $g(S)=1$. The crossing $c$ is non-nugatory implies that $S \cap D$ is an essential arc in $S$ hence by cutting $S$ along $S \cap D$, we get an annulus that bounds $K_0$ (Figure \ref{fig:crossing} (b)). This shows that the two components of $K_0$ are the same knot, say $K''$. Thus by \eqref{eqn:link}
\[ p^{0}_{K_0}(v)=(v^{-2}-1)v^{2lk(K_0)}(p^0_{K''}(v))^2.\]
By the skein relation of the conway polynomial, $lk(K_0) = a_2(K)-a_2(K')$ hence by the skein relation \eqref{eqn:skein} of the zeroth coefficient polynomial we conclude
\[v^{-1}p^{0}_K(v) - v p^{0}_{K'}(v) = (v^{-1}-v)v^{2(a_2(K)-a_2(K'))}(p^0_{K''}(v))^2 \]
as desired.

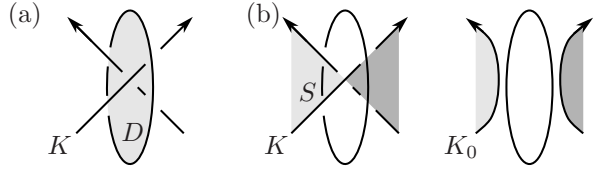
\begin{figure}[htbp]
\begin{center}
\includegraphics*[width=70mm]{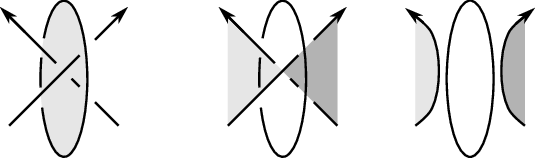}
\begin{picture}(0,0)
\put(-225,55) {(a)}
\put(-182,10) {$D$}
\put(-210,5) {$K$}
\put(-135,55) {(b)}
\put(-115,25) {$S$}
\put(-128,5) {$K$}
\put(-60,5) {$K_0$}
\end{picture}
\caption{(a) crossing disk $D$, (b) The surface $S$ gives rises to an annulus bounding $K_0$.} 
\label{fig:crossing}
\end{center}
\end{figure} 
\end{proof}

A repeated use of Theorem \ref{theorem:Gordian} gives the following obstruction for a genus one knot $K$ to admit a genus non-increasing unknotting sequence of length $n$.
\begin{corollary}
\label{cor:obstruction}
Let $K$ be a genus one knot and let 
\[ K=K_n \stackrel{\varepsilon_n}{\longrightarrow} K_{n-1} \stackrel{\varepsilon_{n-1}}{\longrightarrow} \cdots  \stackrel{\varepsilon_2}{\longrightarrow} K_1 \stackrel{\varepsilon_1}{\longrightarrow} U \]
be a genus non-increasing unknotting sequence of $K$.
Then there exists $k_i \in \Z$ and $f_{i} \in \Z[v^{2},v^{-2}]$ that satisfies the following properties
\begin{itemize}
\item[(i)]
$ \displaystyle p^{0}_K(v) = v^{2\sum_{i=1}^{n}\varepsilon_i} + (1-v^{2}) \sum_{i=1}^{n} \varepsilon_iv^{2k_i}f_i(v)^2 $
\item[(ii)] $f_i(1)=1, f'_i(1)=0$
\end{itemize}
\end{corollary}

The corollary can be used to give a restriction for $u^{g}(K)$ of genus one knots. 
Among them, when we add further restriction that the unknotting sequence is positive, Corollary \ref{cor:obstruction} (i) says that 
$(p^{0}_K(v^2) - v^{2\sum_{i=1}^{n}\varepsilon_i})\slash (1-v^{2})$ is a sum of perfect squares hence it has the following properties.

\begin{lemma}
\label{lemma:positivity}
Let $g(v) = c_m v^m + c_{m+1}v^{m+1} + \cdots + c_M v^{M} \in \Z[v,v^{-1}]$ ($c_m,c_M \neq 0$).
If $g(v)=\sum_{i=1}^{n} f_i(v)^2$ for some $f_i \in \Z[v,v^{-1}]$ and $n \geq 1$, then 
$c_m$ and $c_M$ are sum of at most $n$ perfect squares.
In particular,
\begin{itemize}
\item[(i)] $c_m,c_M>0$.
\item[(ii)] If $c_m$ or $c_M$ is not a perfect square, then $n\geq 2$.
\item[(iii)] If the prime decomposition of $c_m$ or $c_M$ contains a factor of $p^{k}$ with prime $p \equiv 3 \pmod{3}$ and odd $k$, then $n\geq 3$.
\end{itemize}
\end{lemma}
\begin{proof}
This follows from an observation that $c_m$ (resp. $c_M$) is a sum of at most $n$ perfect squares.
\end{proof}

 This positivity phenomenon allows us to relate the HOMFLY polynomial (the  zeroth coefficient polynomial) and $u^{g}_{++}(K)$.

\begin{proof}[Proof of Theorem \ref{theorem:main}]
We put
\[ p^{0}_K(v) = h_{m,0}v^{2m}+ \cdots +h_{M,0}v^{2M}\]
Assume that assume that $K$ has a genus non-increasing positive unknotting sequence of length $n$. Then by Corollary \ref{cor:obstruction} (iii)
\[ \frac{p^{0}_K(v)-v^{2n}}{1-v^{2}} = \sum_{i=1}v^{2k_i} f_i(v)^2 = \sum_{i=1}(v^{k_i}f_i(v))^2.\]

(i-a) If $n < m$, then
\[ p^{0}_K(v)-v^{2n} = -v^{2n} +  h_{m,0}v^{2m}+ \cdots +h_{M,0}v^{2M} \]
so
\begin{align*}
\sum_{i=1}v^{2k_i} f_i(v)^2 &= \frac{p^{0}_K(v)-v^{2n}}{1-v^{2}} = \frac{-v^{2n} + v^{2m}+h_{m,0}v^{2m}+ \cdots +h_{M,0}v^{2M}}{1-v^{2}} \\
&= -v^{2n} + (\deg >2n \mbox{ terms})
\end{align*}
By Lemma \ref{lemma:positivity} this is impossible.\\

(i-b,i-c) By (i-a) we may assume that $n\geq m$. Then
\[ \frac{p^{0}_K(v)-v^{2n}}{1-v^{2}} = h'_{m,0} v^{2m} + \cdots ,\quad h'_{m,0}=  \begin{cases} h_{m,0} & n\neq m \\
h_{m,0}-1 & n=m
\end{cases}
\]
Thus when $h_{m,0}<0$,
 \begin{align*}
\sum_{i=1}v^{2k_i} f_i(v)^2 &= \frac{p^{0}_K(v)-v^{2n}}{1-v^{2}} = h'_{m,0}v^{2m} + (\deg >2m \mbox{ terms})
\end{align*}
for some $h'_{m,0}<0$. This cannot happen by Lemma \ref{lemma:positivity}.
Furthermore if $h_{m,0}=1$ and $n=m$, then 
 \begin{align*}
\sum_{i=1}v^{2k_i} f_i(v)^2 &= \frac{p^{0}_K(v)-v^{2n}}{1-v^{2}} = \frac{h_{m',0}v^{2m'}+ \cdots + h_{M,0}v^{2M}}{1-v^{2}} \\
&= h_{m',0}v^{2m'} + (\deg >2m' \mbox{ terms})
\end{align*}
Since we are assuming $h_{m',0}<0$, this is impossible by Lemma \ref{lemma:positivity}. This proves that $n\geq m+1$.\\

(ii-a) If $n < M$,
\begin{align*}
\sum_{i=1}v^{2k_i} f_i(v)^2 &= \frac{p^{0}_K(v)-v^{2n}}{1-v^{2}} = \frac{ (\deg <2M \mbox{ terms}) + h_{M,0}v^{2M}}{1-v^{2}} \\
&= (\deg <2M \mbox{ terms}) - h_{M,0}v^{2M-2}
\end{align*}
Since $h_{M,0}>0$, this is impossible by Lemma \ref{lemma:positivity} so we conclude $n\geq M$.

(ii-b) By (ii-a) we may assume that $n\geq M$. If $n=M$,
\begin{align*}
\sum_{i=1}v^{2k_i} f_i(v)^2 &= \frac{p^{0}_K(v)-v^{2n}}{1-v^{2}} = \frac{\cdots + (h_{M,0}-1)v^{2M}}{1-v^{2}} \\
&= (\deg <2M \mbox{ terms}) + (-h_{M,0}+1) v^{2M-2}
\end{align*}
since $h_{M,0}>1$, this is impossible by Lemma \ref{lemma:positivity}.

\end{proof}


\begin{thebibliography}{1}
\bibitem[Al]{al}
A.\ Alishahi, 
{\em Unknotting number and Khovanov homology.}
Pacific J. Math.301(2019), no.1, 15--29.

\bibitem[AE]{ae}
A.\ Alishahi and E.\ Eftekhary, 
{\em Knot Floer homology and the unknotting number.}
Geom. Topol.24(2020), no.5, 2435--2469.

\bibitem[BF]{bf}
M.\ Borodzik and S.\ Friedl, 
{\em On the algebraic unknotting number.}
Trans. London Math. Soc.1(2014), no.1, 57--84.

\bibitem[BF2]{bf2}
M.\ Borodzik and S.\ Friedl, 
{\em The unknotting number and classical invariants, I.}
Algebr. Geom. Topol.15(2015), no.1, 85--135.


\bibitem[BH]{bh} M.\ Brittenham and S.\ Hermiller, 
{A counterexample to the Bernhard-Jablan unknotting conjecture.}
Exp. Math.30(2021), no.4, 547--556.

\bibitem[CL]{cl}
T. Cochran and W. B. R. Lickorish.
{\em Unknotting information from 4-manifolds.}
Trans. Amer. Math. Soc.297(1986), no.1, 125--142.

\bibitem[FK]{fl}
P.\ Feller and L.\ Lewark,
{\em On classical upper bounds for slice genera.}
Selecta Math. (N.S.)24(2018), no.5, 4885--4916.

\bibitem[It]{it}
T.\ Ito,
{\em An obstruction of Gordian distance one and cosmetic crossings for genus one knots,}
New York J. Math.28(2022), 175--181.

\bibitem[KL]{kl}
E.\ Kalfagianni and X.-S. Lin,
{\em Knot adjacency, genus and essential tori,}
Pacific J. Math. 228 (2006), 251--275.




\bibitem[Ka]{ka} A. Kawauchi, 
{\em On coefficient polynomials of the skein polynomial of an oriented link,}
Kobe J. Math. 11 (1994), no. 1, 49--68.

\bibitem[KLMMMS]{klmmms}
M.\ Kegel, L.\ Lewark, N.\ Manikandan, F. Misev, L. Mousseau and M. Silvero,
{\em  On unknotting fibered positive knots and braids,}
arXiv:2312.07339v1.


\bibitem[LM]{lm} C.\ Livingston and A.\ Moore, 
{\em KnotInfo: Table of Knot Invariants,}
URL: \url{knotinfo.math.indiana.edu}, (June 21, 2024).

\bibitem[Ow]{ow}
B.\ Owens, 
{\em Unknotting information from Heegaard Floer homology.}
Adv. Math. 217(2008), no.5, 2353--2376.

\bibitem[OS]{os}
P.\ Ozsv\'ath and Z.\ Szab\'o, 
{\em Knots with unknotting number one and Heegaard Floer homology.}
Topology4 4(2005), no.4, 705--745.



\bibitem[Tr]{tr} P.\ Traczyk, 
{\em A criterion for signed unknotting number.}
Low-dimensional topology (Funchal, 1998), 215--220.
Contemp. Math., 233
American Mathematical Society, Providence, RI, 1999.

\end{thebibliography}
\end{document}